\documentclass[a4paper,11pt]{amsart}
\usepackage[colorlinks, linkcolor=blue,anchorcolor=Periwinkle,
    citecolor=blue,urlcolor=Emerald]{hyperref}
\usepackage{graphicx}
\usepackage{amsmath}
\usepackage{amssymb}
\usepackage{oldgerm}
\usepackage[all]{xy}
\usepackage{tikz}
\usepackage{tikz-cd}

\newcommand{\End}{\operatorname{End}\nolimits}
\newcommand{\Hom}{\operatorname{Hom}\nolimits}
\newcommand{\RHom}{\operatorname{RHom}\nolimits}

\newtheorem{theo}{Theorem}[section]

\newtheorem{cor}[theo]{Corollary}

\newcommand{\ten}{\otimes}
\newcommand{\DMix}{\mathrm{DMix}}
\newcommand{\lten}{\overset{\mathbf{L}}{\ten}}

\newcommand{\iso}{\stackrel{_\sim}{\rightarrow}}

\newcommand{\ko}{\;\; ,}

\newcommand{\ALG}{\textbf{ALG}}
\newcommand{\Alg}{\textbf{Alg}}
\newcommand{\rep}{\mathrm{rep}}

\renewcommand{\phi}{\varphi}

\begin{document}
\baselineskip=15pt
\title[Derived invariance of the Tamarkin--Tsygan calculus]{Derived invariance of
the Tamarkin--Tsygan calculus of an algebra}

\author{Marco Antonio Armenta}
\address{CIMAT A. C. Guanajuato, M\'exico}
\address{IMAG, Univ Montpellier, CNRS, Montpellier, France.}

\email{drmarco@cimat.mx}

\author{Bernhard Keller}
\address{Universit\'e Paris Diderot -- Paris 7\\
    Sorbonne Universit\'e\\
    UFR de Math\'ematiques\\
    CNRS\\
   Institut de Math\'ematiques de Jussieu--Paris Rive Gauche, IMJ-PRG \\   
    B\^{a}timent Sophie Germain\\
    75205 Paris Cedex 13\\
    France
}
\email{bernhard.keller@imj-prg.fr}
\urladdr{https://webusers.imj-prg.fr/~bernhard.keller/}

%\date{\today}

 \subjclass[2000]{}

\begin{abstract} We prove that derived equivalent algebras have isomorphic
differential calculi in the sense of Tamarkin--Tsygan.
\end{abstract}

\maketitle

\section{Introduction}
Let $k$ be a commutative ring and $A$ an associative $k$-algebra projective
as a module over $k$. We write $\otimes$ for the tensor product over $k$. We point out that all the constructions and proofs of this paper extend to small dg categories cofibrant
over $k$. The Hochschild homology $HH_\bullet(A)$ and cohomology $HH^\bullet(A)$ are derived invariants of $A$, see \cite{Happel88, Happel89, Rickard89a, Rickard91,
Zimmermann07}. Moreover, these $k$-modules come with operations, namely the cup product
\[
    \cup :HH^n(A) \otimes HH^m(A) \to HH^{n+m}(A),
\]
the Gerstenhaber bracket
\[
    [-,-]:HH^n(A) \otimes HH^m(A) \to HH^{n+m-1}(A),
\]
the cap product
\[
    \cap:HH_n(A) \otimes HH^m(A) \to HH_{n-m}(A)
\]
and Connes' differential
\[
    B:HH_n(A) \to HH_{n+1}(A),
\]
such that $B^2=0$ and 
\begin{equation} \label{eq:TamarkinTsygan}
 [B i_\alpha-(-1)^{|\alpha|}i_\alpha B, i_\beta] = i_{[\alpha,\beta]},
\end{equation}
where $i_\alpha(z)=(-1)^{|\alpha||z|} z \cap \alpha$. This is the first example \cite{GelfandDaletskiiTsygan89, TamarkinTsygan} of a \textit{differential calculus} or a \textit{Tamarkin-Tsygan calculus}, which is by definition a collection
\[
(\mathcal{H}^\bullet,\cup,[-,-],\mathcal{H}_\bullet, \cap, B),
\]
such that $(\mathcal{H}^\bullet,\cup,[-,-])$ is a Gerstenhaber algebra, the cap product $\cap$ endows $\mathcal{H}_\bullet$ with the structure of a graded Lie module over the Lie algebra $(\mathcal{H}^\bullet[1],\cup,[-,-])$ and the map $B:\mathcal{H}_n \to \mathcal{H}_{n+1}$ squares to zero and satisfies the equation~(\ref{eq:TamarkinTsygan}).
The Gerstenhaber algebra $(HH^\bullet(A),\cup,[-,-])$ has been proved to be a derived invariant \cite{Keller04, Keller03}. The cap product is also a derived invariant \cite{ArmentaKeller17}. In this work, we use an isomorphism induced from the cyclic functor \cite{Keller98} to prove derived invariance of Connes' differential and of the ISB-sequence. To obtain derived invariance of the differential calculus, we need to prove that this isomorphism equals the isomorphism between Hochschild homologies used in \cite{ArmentaKeller17} to prove derived invariance of the cap product.

\section{The cyclic functor}

Let $\Alg$ be the category whose objects are the associative dg (=differential graded)
$k$-algebras cofibrant over $k$ (i.e. `closed' in the sense of section~7.5 of \cite{Keller98}) and
whose morphisms are morphisms of dg $k$-algebras which do not necessarily preserve the unit.
Let $\rep(A,B)$ be the full subcategory of the derived category $D(A^{op}\ten B)$ whose
objects are the dg bimodules $X$ such that the restriction $X_B$ is compact in $D(B)$, i.e.
lies in the thick subcategory generated by the free module $B_B$.
Define $\ALG$ to be the category whose objects are those of $\Alg$ and whose 
morphisms from $A$ to $B$ are the isomorphism classes in $\rep(A,B)$. The composition of morphisms in $\ALG$ is given by the total derived tensor product \cite{Keller98}. The identity of 
$A$ is the isomorphism class of the bimodule ${_AA_A}$. There is a canonical functor 
$\Alg \to \ALG$ that associates to a morphism $f:A \to B$ the bimodule ${_fB_B}$ with 
underlying space $f(1)B$ and 
$A$-$B$-action given by $a.f(1)b.b'=f(a)bb'$.

Let $\Lambda$ be the dg algebra $k[\epsilon]/(\epsilon^2)$ where $|\epsilon|=-1$
and the differential vanishes. As in 
\cite{Kassel, Keller98}, we will identify the category of dg $\Lambda$-modules with the category of mixed complexes. Denote by $\DMix$ the derived category of dg $\Lambda$-modules. 
Let $C:\Alg \to \DMix$ be the \textit{cyclic functor} \cite{Keller98}, that is, the underlying dg $k$-module of $C(A)$ is the mapping cone over $(1-t)$ viewed as a morphism of complexes 
$(A^{\otimes \ast+1},b') \to (A^{\otimes \ast},b)$ and the first and second differentials of the mixed complex $C(A)$ are
\[
\left[ \begin{array}{cc}
    b & 1-t \\
    0 & -b'
\end{array} \right] 
\]
and 
\[ \left[ \begin{array}{cc}
    0 & 0 \\
    N & 0
\end{array} \right].
\]
Clearly, a dg algebra morphism $f:A \to B$ (even if it does not preserve the unit)
induces a morphism $C(f):C(A) \to C(B)$ of dg $\Lambda$-modules. 
Let $X$ be an object of $\rep(A,B)$. We assume, as we may, that $X$ is cofibrant
(i.e. `closed' in the sense of section~7.5 of \cite{Keller98}). This implies that $X_B$ is
cofibrant as a dg $B$-module and thus that morphism spaces in the derived category
with source $X_B$ are isomorphic to the corresponding morphism spaces in the
homotopy category. Consider the morphisms
\[
    \xymatrix{ A \ar[rr]^-{\alpha_X} & & \End_B(B\oplus X) & & B \ar[ll]_-{\beta_X} 
    }
\]
where $\End_B(B\oplus X)$ is the differential graded endomorphism algebra
of $B\oplus X$, the morphism $\alpha_X$ be given by the left action of $A$ on $X$ and $\beta_X$ is 
induced by the left action of $B$ on $B$. Note that these morphisms do not preserve the units.
The second author proved in \cite{Keller98} that $C(\beta_X)$ is invertible in $\DMix$ and defined $C(X)=C(\beta_X)^{-1} \circ C(\alpha_X)$.
We recall that $C$ is well defined on $\ALG$ and that this extension
of $C$ from $\Alg$ to $\ALG$ is unique by Theorem~2.4 of \cite{Keller98}. 

Let $X:A \to B$ be a morphism of $\ALG$ where $X$ is cofibrant. Put $X^\vee=\Hom_B(X,B)$.
We can choose morphisms 
$u_X: A \to X \lten_B X^\vee$ and $v_X:X^\vee \lten_B X \to B$ such that the following triangles commute
\[
\xymatrix{ X \ar[r]^-{u_X \otimes 1} \ar[rd]_-= & X \lten_B X^\vee \lten_A X \ar[d]^-{1 \otimes v_X} \\
& X \\
} \quad 
    \xymatrix{ X^\vee \ar[r]^-{1 \otimes u_X} \ar[rd]_-= & X^\vee \lten_A X \lten_B X^\vee \ar[d]^-{v_X \otimes 1} \\
& X^\vee. \\
}
\]
Then the functors 
\[
? \lten_{A^e} (X \otimes X^\vee) : D(A^e) \to D(B^e)
\]
and 
\[
?\lten_{B^e} (X^\vee \otimes X):D(B^e) \to D(A^e)
\]
form an adjoint pair. We will identify $X \lten_B X^\vee \iso (X \otimes X^\vee) \lten_{B^e} B$ and $X^\vee \lten_A X \iso (X^\vee \otimes X) \lten_{A^e} A$, and still call $u_X$ and $v_X$ the same morphisms when composed with this identification. Since $k$ is a commutative ring, the tensor product over $k$ is symmetric. We will denote the symmetry isomorphism by $\tau$. Let
$D(k)$ denote the derived category of $k$-modules. We define a functor $\psi:\Alg \to D(k)$ by putting $\psi(A)=A \lten_{A^e} A$, and $\psi(f)=f\otimes f$ for a morphism $f:A \to B$. There is a canonical quasi-isomorphism $\psi(A) \to \varphi(A)$ 
for any algebra $A$, where $\varphi(A)$ is the underlying complex of $C(A)$.
Therefore, the functors $\varphi$ and $\psi$ take isomorphic values on objects. We now
define $\psi$ on morphisms of $\ALG$ as follows: Let $X$ be a cofibrant object of $\rep(A,B)$.
Define $\psi(X)$ to be the composition
\begin{eqnarray*}
A \lten_{A^e} A & \to & A \lten_{A^e} X \otimes X^\vee \lten_{B^e} B \\
 & \iso & B \lten_{B^e} X^\vee \otimes X \lten_{A^e} A \\
 & \to & B \lten_{B^e} B. \\
\end{eqnarray*}
That is, we put $\psi(X)=(1 \otimes v_X)\circ\tau\circ(1 \otimes u_X)$. 

\begin{theo} \label{thm:main}
The assignments $A \mapsto \psi(A)$, $X \mapsto \psi(X)$ define a functor on $\ALG$
that extends the functor $\varphi:\Alg \to D(k)$.
\end{theo}

\begin{cor} The functors $\varphi$ and $\psi: \ALG \to D(k)$ are isomorphic.
\end{cor}

\begin{proof}[Proof of the Corollary] This is immediate from Theorem~2.4 of \cite{Keller98} and the remark following it.
\end{proof}

\begin{proof}[Proof of the Theorem] For ease of notation, we write $\ten$ and $\Hom$
instead of $\lten$ and $\RHom$.
Let $f:A \to B$ be a morphism of $\Alg$. The associated morphism in $\ALG$ is $X={_fB_B}$. 
Note that $X^\vee={_BB_f}$. The diagrams
\[
\xymatrix{ A \ten_{A^e} ( _fB \otimes_B B_f) \ar[rd]^-{\simeq} \ar[d]_-{\simeq} & \\
A \ten_{A^e} ( _fB \otimes B_f) \ten_{B^e} B \ar[r]_-{\simeq} & A \ten_{A^e} {_fB_f} }
\]
and
\[
\xymatrix{ A \ten_{A^e} ( _fB \otimes B_f) \ten_{B^e} B \ar[rr]_-{\simeq} \ar[d]_-{\tau} & & A \ten_{A^e} {_fB_f} \ar[d]^-{\tau} \\
B \ten_{B^e} B_f \otimes {_fB} \ten_{A^e} A \ar[rr]^-{\simeq} & & {_fB_f} \ten_{A^e} A
}
\]
are commutative. Since
\[
\xymatrix{ {_fB_f} \ten_{A^e} A \ar[r]^{\tau} \ar[d]_-{1 \otimes f}   & 
						A \ten_{A^e} {_fB_f} \ar[d]^-{f\otimes 1}  \\
{B} \ten_{B^e} {B}   \ar[r]^-\tau & B\ten_{B^e} B\\
}
\]
is also commutative and the bottom morphism equals the identity, 
we get that $\psi({_fB_B})$ is the morphism induced by $f$ from $A \ten_{A^e} A$ to $B\ten_{B^e} B$. Therefore $\psi({_fB_B})=\varphi({_fB_B})$.
Let $X:A\to B$ and $Y:B \to C$ be morphisms in $\ALG$.
We have canonical isomorphisms
\begin{align*}
\Hom_C(Y,C) \ten_B \Hom_B(X,B)   & \iso \Hom_B(X,\Hom_C(Y,C)) \\
                            & \iso \Hom_C(X\ten_B Y, C).
\end{align*}
Whence the identification
\[
(X\ten_B Y)^\vee = Y^\vee\ten_B X^\vee.
\]
Put $Z=X\ten_B Y$. For $u_Z$, we choose the composition 
\[
\xymatrix{
A\ar[r]^-{u_X} & X \ten_B X^\vee \ar[rr]^-{1\ten u_Y\ten 1} & &
X\ten_B Y \ten_C \ten Y^\vee \ten_B X^\vee}
\]
and for $v_Z$ the composition
\[
\xymatrix{
(Y^\vee \ten_B X^\vee) \ten_A (X \ten_B Y) \ar[rr]^-{1\ten v_X \ten 1} & &
Y^\vee\ten_B Y \ar[r]^-{v_Y} & C}.
\]
By definition, the composition $\psi(Y)\circ \psi(X)$ is the composition of
$(1\ten v_Y)\circ \tau \circ (1\ten u_Y)$ with $(1\ten v_X)\circ \tau\circ (1\ten u_X)$.
We first examine the composition $(1\ten u_Y)\circ (1\ten v_X)$:
\[
\xymatrix{
B \ten_{B^e}(X^\vee \ten X)\ten_{A^e} A \ar[r]^-{1\ten v_X} & B\ten_{B^e} B \ar[r]^-{1\ten u_Y} &
B\ten_{B^e}(Y\ten Y^\vee) \ten_{C^e} C}
\]
Clearly, the following square is commutative
\[
\xymatrix{
B\ten_{B^e}(X^\vee\ten X) \ten_{A^e} A \ar[r]^-c  \ar[d]_{1\ten v_X}
	& ((X^\vee \ten X)\ten_{A^e} A)\ten_{B^e} B \ar[d]^{v_X\ten 1} \\
B\ten_{B^e} B \ar[r]^-\tau & B\ten_{B^e} B\ko}
\]
where $c$ is the obvious cyclic permutation. Notice that 
\[
\tau: B\ten_{B^e} B \to B\ten_{B^e}B
\]
equals the identity. Thus, we have $1\ten u_Y=(1\ten u_Y)\circ \tau$ and
\[
(1\ten u_Y)\circ (1\ten v_X) = (1\ten u_Y)\circ \tau \circ (1\ten v_X) =
(1\ten u_Y) \circ (v_X\ten 1) \circ c.
\]
Let $\sigma$
\[
((X^\vee\ten X)\ten_{A^e} A) \ten_{B^e} (Y\ten Y^\vee) \ten_{C^e} C \iso
A \ten_{A^e}(X\ten_B Y) \ten (Y^\vee \ten_B X^\vee) \ten_{C^e} C
\]
be the natural isomorphism given by reordering the factors. Then we have
$\psi(Y)\circ \psi(X)=f\circ g$, where $f=\sigma\circ (1\ten u_Y) \circ c\circ \tau \circ (1\ten u_X)$
and $g=( v_Y \ten 1) \circ \tau \circ (v_X\ten 1)\circ \sigma^{-1}$. It is not hard to see
that $f$ equals $1\ten u_Z$ and $g$ equals $(1\ten v_Z)\circ \tau$. Intuitively, the reason
is that given the available data, there is only one way to go from $A\ten_{A^e}A$ to
\[
A \ten_{A^e}(X\ten_B Y) \ten (Y^\vee \ten_B X^\vee) \ten_{C^e} C
\]
and only one way to go from here to $C\ten_{C^e} C$. It follows that
$\psi(Y)\circ\psi(X)=\psi(Z)$.
\end{proof}

\section{Derived invariance}

Let $A$ and $B$ be derived equivalent algebras and $X$ a cofibrant object of
$\rep(A,B)$ such that $?\lten_A X: D(A) \to D(B)$ is an equivalence. Then $C(X)$
is an isomorphism of $\DMix$ and $\varphi(X)$ an isomorphism of $D(k)$.
There is a canonical short exact sequence of dg $\Lambda$-modules
\[
    0 \to k[1] \to \Lambda \to k \to 0
\]
giving rise to a triangle
\[
\xymatrix{
k[1] \ar[r]^-{B'} & \Lambda \ar[r]^I & k \ar[r]^-S & k[2]}.
\]

We apply the isomorphism of functors $?\lten_\Lambda C(A) \iso ?\lten_\Lambda C(B) $ 
to this triangle to get an isomorphism of triangles in $D(k)$, where
we recall that $\phi(A)$ is the underlying complex of $C(A)$
\[
\xymatrix{k[1] \lten_\Lambda C(A) \ar[r]^-{B'} \ar[d]_{\cong} & \phi(A) \ar[r]^-{I} \ar[d]_{\phi(X)} & k \lten_\Lambda C(A) \ar[r]^-{S} \ar[d]_{\cong} & k[2] \lten_\Lambda C(A) \ar[d]_{\cong} \\
k[1] \lten_\Lambda C(B) \ar[r]^-{B'} & \phi(B) \ar[r]^-{I} & k \lten_\Lambda C(B) \ar[r]^-{S} & k[2] \lten_\Lambda C(B)
}
\]
Taking homology and identifying $H_j(k\lten_\Lambda C(A))= HC_j(A)$ as in \cite{Kassel}, 
gives an isomorphism of the ISB-sequences of $A$ and $B$,
\[
    \xymatrix{ \cdots \ar[r] & HC_{n-1}(A) \ar[r]^-{B'_{n-1}} \ar[d]_-{\cong} & HH_n(A) \ar[r]^-{I_{n}} \ar[d]^-{HH_n(X)} & HC_n(A) \ar[r]^-{S_{n}} \ar[d]_-{\cong} & HC_{n-2}(A) \ar[r] \ar[d]_-{\cong} & \cdots \\
    \cdots \ar[r] & HC_{n-1}(B) \ar[r]^-{B'_{n-1}} & HH_n(B) \ar[r]^-{I_{n}} & HC_n(B) \ar[r]^-{S_{n}} & HC_{n-2}(B) \ar[r] & \cdots\ko
    }
\]
where $HH_n(X)$ is the map induced by $\phi(X)$. In terms of the differential calculus, Connes' differential is the map
\[
    B_n:HH_n(A) \to HH_{n+1}(A),
\]
given by $B_n=B'_n I_n$. This shows that $B_n$ is derived invariant via 
$HH_n(X)$. By Theorem~\ref{thm:main}, the map $HH_n(X)$ is equal
to the map induced by $\psi(X)$ used in the proof of the derived
invariance of the cap product \cite{ArmentaKeller17}. Therefore, we
get the following

\begin{theo}
The differential calculus of an algebra is a derived invariant.
\end{theo}

\bibliographystyle{amsplain}
\bibliography{references}
\newpage
\end{document}